\documentclass[11pt]{amsart}
\usepackage{amssymb,mathrsfs,graphicx,enumerate}
\usepackage{amsmath,amsfonts,amssymb,amscd,amsthm,bbm}
\usepackage{graphicx,colortbl}
\usepackage{mathtools}

\DeclareMathOperator*{\essinf}{essinf}

\title[A discrete consensus optimization algorithm]{Convergence and error estimates for time-discrete consensus-based optimization algorithms}

\author[Ha]{Seung-Yeal Ha}
\address[Seung-Yeal Ha]{\newline Department of Mathematical Sciences and Research Institute of Mathematics, \newline
Seoul National University, Seoul, 08826 and \newline
Korea Institute for Advanced Study, Hoegiro 85, Seoul, 02455, Korea (Republic of)}
\email{syha@snu.ac.kr}

\author[Jin]{Shi Jin}
\address[Shi Jin]{\newline School of Mathematical Sciences, MOE-LSC, and Institute of Natural Sciences, \newline Shanghai Jiao Tong University, Shanghai 200240, China}
\email{shijin-m@sjtu.edu.cn}

\author[Kim]{Doheon Kim}
\address[Doheon Kim]{\newline School of Mathematics, Korea Institute for Advanced Study,\newline Hoegiro 85, Seoul 02455, Korea (Republic of)}
\email{doheonkim@kias.re.kr}

\newtheorem{theorem}{Theorem}[section]
\newtheorem{lemma}{Lemma}[section]
\newtheorem{corollary}{Corollary}[section]
\newtheorem{proposition}{Proposition}[section]
\newtheorem{remark}{Remark}[section]

\newtheorem{definition}{Definition}[section]

\newcommand{\bbr}{\mathbb R}

\newcommand{\bbe}{\mathbb E}

\newcommand{\bbp} {\mathbb P}

\begin{document}
%%%%%%%%%%%%%%%%
%\thanks{* Corresponding author.}
\date{\today}


\thanks{\textbf{Acknowledgment.} The work of S.-Y. Ha was supported by National Research Foundation of Korea (NRF-2020R1A2C3A01003881), and the work of S. Jin was supported by NSFC Grant Nos. 11871297 and  3157107, and the work of D. Kim was supported by a KIAS Individual Grant (MG073901) at Korea Institute for Advanced Study. The authors would like to thank to Dr. Dongnam Ko for his helpful comments on the Laplace principle}

\begin{abstract}
We present convergence and error estimates of the time-discrete consensus-based optimization(CBO) algorithms proposed in \cite{C-J-L-Z} for general nonconvex functions. In authors' recent work \cite{H-J-K}, rigorous error analysis of the first-order consensus-based optimization algorithm proposed in \cite{C-J-L-Z} was studied at the particle level without resorting to the kinetic equation via a mean-field limit. However, the error analysis for the corresponding time- discrete algorithm was not done mainly due to lack of discrete analogue of It\^o's stochastic calculus. In this paper, we provide a simple and elementary convergence and error analysis for a general time-discrete consensus-based optimization algorithm, which  includes the three  discrete algorithms in \cite{C-J-L-Z}. Our analysis provides numerical stability and convergence conditions for the three algorithms, as well as
error estimates to the global minimum. 
\end{abstract}
\maketitle \centerline{\date}

%\tableofcontents

\section{Introduction} \label{sec:1}
\setcounter{equation}{0}
The purpose of this work is to complete the convergence and error analysis for discrete consensus-based optimization algorithms introduced in \cite{C-J-L-Z, H-J-K}.  For modern  machine learning methods, one needs to solve
non-convex optimiation problems in high dimensions. It is well-known that non-convex optimiation problem is NP-hard. Usually a deterministic algorithm, such as the gradient descent method, will get stuck to local
minima. In order to escape from  local minima, or saddle points, one needs to introduce some numerical noises which allow the algorithms to escape  from the local minima or saddle points.
For this reason, stochastic optimizations--such as the stochastic gradient descent method, have been widely
used in machine learning \cite{Be}.  On the other hand, gradient-free algorithms, which do not need the gradient of objective functions, are attractive for problems with non-smooth objective functions
or data-based optimization problems.  Meta-heuristic stochastic optimization algorithms belong to the latter category, for example  swarm intelligence methods \cite{Ke, Y-D} such as particle swarm optimization (in short PSO) \cite{E-K}, simulated annealing method \cite{K-G-V, L-A}, ant-colony algorithm \cite{Ya}, genetic algorithm \cite{Ho}  etc. 
A basic idea of these meta-heuristic algorithms is to use collective behaviors of underlying individual agents (or particles)  coupled with suitable stochastic components in the choice of system parameters. 
Although each individual moves in some random fashion, one designs suitable communication functions between the particles such that colelctively they exhibit some intelligent behavior, such as moving toward the
global minimum. Although these algorithms  are usually simple to implement and yields  reasonably good results with suitable choices of parameters, their rigorous convergence analysis were mostly open as far as the authors know. 

Our main interest in this paper lies on the first-order consensus based optimization algorithm, first introduced in \cite{P-T-T-M} and studied in \cite{C-C-T-T}, and then modified in  \cite{C-J-L-Z} (see also \cite{T-P-B-S} for comparisons between CBO algorithm and other heuristic algorithms based on collective dynamics). These are
swarming intellience models that can be proved to exhibit concensus approximating the global minimum of  general nonconvex functions under suitable conditions on the parameters and initial data.   We also refer to recent works on the consensus-based optimization algorithm on the sphere \cite{F-H-P-S1, F-H-P-S2}.  

We begin with the continuous optimization algorithm. Let $X_t^k = (x_t^{k,1}, \cdots, x_t^{k,d}) \in \bbr^d$ be the coordinate  of the $k$-th particle at time $t$, and $L = L(X),~X \in \bbr^d$ be a {\it non-convex} objective function to be minimized. Then, the main goal of optimization algorithms is to look for a global minimizer $X^*$ of $L$ in the search space (in our setting, the whole space) if it exists:
\[  \displaystyle  X^* \in \mbox{argmin}_{X \in \bbr^d} L(X). \]
 In a recent work \cite{C-J-L-Z}, the authors proposed the following variant of the CBO algorithm introduced in \cite{C-C-T-T, P-T-T-M}: 
\begin{equation} \label{A-1}
\begin{cases}
\displaystyle dX^i_t = -\lambda (X^i_t - {\bar X}_t^*)dt  + \sigma  \sum_{l=1}^{d} (x^{i,l}_t - {\bar x}_t^{*,l}) dW_t^l e_l, \quad t > 0,~~ i = 1, \cdots, N, \\
\displaystyle {\bar X}_t^* =(x_t^{*,1}, \cdots, x_t^{*,d})  := \frac{\sum_{j=1}^{N} X^j_t e^{-\beta L(X^j_t)}}{\sum_{j=1}^{N} e^{-\beta L(X^j_t)}},
\end{cases}
\end{equation}
where $\lambda$ and $\sigma$ denote the drift rate and noise intensity, respectively, and  $\beta > 0$ is a positive constant corresponding  
to the reciprocal of temperature in statistical physics. Here $\{e_l \}$ is the standard orthonormal basis in $\bbr^d$. The one-dimensional Brownian motions $W_t^l$ are assumed to be i.i.d. and satisfy the mean zero and covariance relations:
\[ \bbe[W_t^l]= 0 \quad \mbox{for $l = 1, \cdots, d$} \quad  \mbox{and} \quad  \bbe[W_t^{l_1} W_t^{l_2}] = \delta_{l_1 l_2} t, \quad 1 \leq  l_1, l_2 \leq d. \]
This model is an example of agent-based swarming models which have been studied intensively  in recent years, see for example several survey articles \cite{A-B, A-B-F, C-H-L, P-R-K, V-Z} and related literature \cite{C-S, F-H-J, H-L, H-L-L,K-C-B-F-L, Ku1, Ku2, M-T, Pe}. \newline

Next, we consider time-discrete analogue of \eqref{A-1}. For this, we set 
\[ h := \Delta t, \quad X_n := X(nh), \quad n = 0, 1, \cdots, \cdots.\]
Then the discrete scheme reads as follows:
\begin{equation}  \label{A-2}
\begin{cases}
\displaystyle X^i_{n+1} =X^i_n -\gamma  (X^i_n - {\bar X}_n^*)  -\sum_{l=1}^{d} (x^{i,l}_n - {\bar x}_n^{*,l})  \eta_{n}^l e_l,\quad n\geq 0,~~ i = 1, \cdots, N, \\
\displaystyle {\bar X}_n^* =(x_n^{*,1}, \cdots, x_n^{*,d})  := \frac{\sum_{j=1}^{N} X^j_n e^{-\beta L(X^j_n)}}{\sum_{j=1}^{N} e^{-\beta L(X^j_n)}},
\end{cases}
\end{equation}
where the random variables $\{ \eta^l_n \}_{n, l}$ are i.i.d. with
\begin{equation} \label{A-3}
\mathbb E[\eta_n^l]=0, \quad \mathbb E[|\eta_n^l|^2]=\zeta^2, \quad n = 1, \cdots, \quad l = 1, \cdots, d.
\end{equation}
Note that the discrete process  $\{ \eta^l_n \}_{n, l}$ certainly includes the Gaussian noise process. In this sense, the discrete model \eqref{A-2} clearly generalizes the discrete model studied in \cite{H-J-K}.  \newline

In this paper, we are interested in the following issues for the discrete algorithm \eqref{A-2}: \newline
\begin{itemize}
\item
(Question A):~Does the $N$-state ensemble $\{ X_n^i \}$ exhibit a global consensus? i.e., does
\[   X_n^i - X_n^j  \to  0\quad\mbox{as}\quad n \to \infty, \quad i, j = 1, \cdots, N \quad \mbox{in suitable sense}? \]

\vspace{0.2cm}

\item
(Question B):~If the answer to the first problem is positive, then under what conditions on system parameters and initial data, does there exist a global consensus state $X_\infty$ such that 
\[ X_n^i \to X_\infty \quad \mbox{for all }i, \quad \mbox{as} ~~n \to \infty,  \quad  \mbox{such that}~~ L(X_\infty) \sim \min_{X} L(X).  \]
\end{itemize}
In \cite{H-J-K}, the above two questions were answered positively for the continuous algorithm \eqref{A-1}, whereas for the discrete algorithm \eqref{A-2}, only the first question was discussed in the same paper. 
For a rigorous error analysis to the continuous algorithm, It\^o's calculus was essentially used. The main reason that we did not cover the second question for the discrete algorithm is mainly due to the lack of the discrete analogue of It\^o's stochastic calculus. 
  \newline

In this paper, we revisit the second question on the convergence of the discrete algorithm \eqref{A-2} and provide positive answer for the generalized discrete algorithm \eqref{A-2} which can cover several discrete algorithms proposed in \cite{C-J-L-Z}  (see Section \ref{sec:2.1} for details).
Our analysis also provides numerical stability conditions and error estimates or these algorithms \newline

Next, we summarize our main results as follows. First, we provide several stochastic global consensus results in suitable sense. More precisely, if system parameters $\gamma$ and $\zeta$ in \eqref{A-2} and \eqref{A-3} satisfy 
\begin{equation} \label{NN-1}
|1-\gamma| < 1, \quad 0\leq \zeta\leq\infty , 
\end{equation}
then the expectation of $X_n^i - X_n^j$ tends to zero asymptotically:
\[ \lim_{n \to \infty} \bbe[X_n^i - X_n^j] = 0, \quad \forall~i, j = 1, \cdots, N. \]
On the other hand, if system parameters satisfy a more restricted condition compared to \eqref{NN-1}:
\begin{equation} \label{A-3-1}
 (1-\gamma)^2 + \zeta^2 < 1,
 \end{equation}
then  $L^2$-global consensus and almost-sure global consensus occur asymptotically:
\[ \lim_{n \to \infty} \bbe |X_n^i - X_n^j |^2 = 0 \quad \mbox{and} \quad \bbp \Big\{  \lim_{n \to \infty} |X_n^i - X_n^j| = 0 \Big \} = 1, \quad   \forall~i, j = 1, \cdots, N,  \]
where $|\cdot|$ stands for the $L^2$-norm of a vector.

Moreover, the above $L^2$-global consensus also implies the $L^1$-global consensus. We refer to Definition \ref{D2.1} and Theorem \ref{T2.1} for the definition of convergences and details.   

Note that the above global consensus does not imply the existence of asymptotic consensus state $X_\infty$ independent of $i$ such that 
\begin{equation} \label{A-4}
 X_n^i \to X_\infty \quad \mbox{in suitable sense}, 
 \end{equation}
i.e., the process can fluctuate, but the relative distances tends to zero asymptotically, for example, the sample paths may tend to periodic orbit or limit cycle. Our second result provides  the condition under which  the process tends to a common fixed random variable $X_\infty$  (see \eqref{A-4}). In fact, under the same assumption \eqref{A-3-1}, one can show that there exists a common constant state $X_\infty=(x_\infty^1,\cdots,x_\infty^d)$ such that
\[  \lim\limits_{n\to\infty}  X_n^i=X_\infty \quad \mbox{a.s.},~1\leq i\leq N. \]
(see Theorem \ref{T3.1} for details). 

Finally, our last result establishes  the condition under which the asymptotic state $X_\infty$ lies in a small neighborhood of a unique global minimum $X_m$ for a large $\beta$. More precisely, under the assumption \eqref{A-3-1} and for a well-prepared initial random variable $X_{in}$ such that $X_n^i \sim X_{in}$, we derive a key estimate (see Proposition \ref{P3.1}):
\[
\essinf_{\omega\in\Omega} L(X_{\infty})\leq-\frac{1}{\beta}\log \mathbb E e^{-\beta L(X_{in})}-\frac{1}{\beta}\log\varepsilon.
\]
Then, by Laplace's method (\cite{Hsu}), one can derive 
\[
\essinf\limits_{\omega\in\Omega} L(X_\infty)\leq L(X_*)+ \frac{d}{2}\frac{\log\beta}{\beta}+E(\beta),
\]
for some function $E(\beta)= O\left(\frac{1}{\beta} \right),~~\beta\gg 1$. We refer to Remark \ref{R3.1} for the intriguing relation between $\beta$ and admissible reference random variable $X_{in}$. \newline

The rest of this paper is organized as follows. In Section \ref{sec:2}, we provide reductions of previously  studied discrete algorithms to \eqref{A-2} - \eqref{A-3}, and then study our first set of main results on the global consensus. In Section \ref{sec:3}, we study the emergence of a global consensus state and provide error estimates toward the global minimum for \eqref{A-2}. Finally, Section \ref{sec:4} is devoted to a brief summary of our main results and some remaining issues to be explored in a future work. \newline

\noindent {\bf Notation}. For a random variable $Z \in \bbr$ on the probability space $(\Omega, {\mathcal F}, \bbp)$, we denote its mean by $\bbe Z$ or $\bbe[Z]$ interchangeably.%, and the function space ${\mathcal C}_b^2(\bbr^d)$ denotes the collection of all ${\mathcal C}^2(\bbr^d)$ functions with bounded derivatives up to second-order. 

\section{Emergence of global consensus} \label{sec:2}
\setcounter{equation}{0}
In this section, we first show that several discrete algorithms introduced in \cite{C-J-L-Z, H-J-K} can be reduced to our generalized discrete algorithm \eqref{A-2} and then we provide a sufficient framework leading to the global consensus for the discrete optimization algorithm in terms of system parameters and initial data.

\subsection{Discrete algorithms} \label{sec:2.1} In this subsection, we show how the previous discrete algorithms studied can be reduced as special cases for our general discrete model. \newline

Let $X_n^k = (x_n^{k,1}, \cdots, x_n^{k,d}) \in \bbr^d$ be the position of the $k$-th particle at time $n$ $(1\leq k\leq N)$. Suppose that the function $L:\mathbb R^d\to\mathbb R$ has exactly one global minimum, and we consider the general discrete consensus-based optimization algorithm:
\begin{equation} \label{B-1}
\begin{cases}
\displaystyle X^i_{n+1} =X^i_n -\gamma (X^i_n - {\bar X}_n^*)  -\sum_{l=1}^{d} (x^{i,l}_n - {\bar x}_n^{*,l}) \eta_n^l e_l,\quad n\geq 0, \quad i = 1, \cdots, N, \\
\displaystyle {\bar X}_n^* =(x_n^{*,1}, \cdots, x_n^{*,d})  := \frac{\sum_{j=1}^{N} X^j_n e^{-\beta L(X^j_n)}}{\sum_{j=1}^{N} e^{-\beta L(X^j_n)}},
\end{cases}
\end{equation}
where $\{e_l \}$ is the standard orthonormal basis in $\bbr^d$ and random variables $\{ \eta^l_n \}_{n, l}$ are i.i.d with
\begin{equation} \label{B-2}
\mathbb E[\eta_n^l]=0 \quad \mbox{and} \quad \mathbb E|\eta_n^l|^2=\zeta^2.
\end{equation}
In the sequel, we consider the following three discrete algorithms. \newline

\noindent $\bullet$~Model A:  Consider the first-order Euler type discrete model in \cite{H-J-K}:
\begin{equation*} \label{B-3}
X^i_{n+1} =X_n^i -\lambda h (X^i_n - {\bar X}_n^*)  -  \sum_{l=1}^{d} (x^{i,l}_n - {\bar x}_n^{*,l}) \sigma \sqrt{h}Z_n^l e_l, \quad n\geq 0,~~ i = 1, \cdots, N,
\end{equation*}
where the random variables $\{ Z^l_n \}_{n, l}$ are i.i.d  standard normal distributions, i.e. $Z^l_n~\sim~{\mathcal N}(0, 1^2)$. If we set 
\begin{equation} \label{B-3-1}
\gamma := \lambda h \quad \mbox{and} \quad \eta_n^l := \sigma \sqrt{h}Z_n^l. 
\end{equation}
Then, the above setting clearly satisfies the relations \eqref{B-2} with $\zeta = \sigma \sqrt{h}$.

\vspace{0.5cm}

\noindent $\bullet$~Model B: Consider a predictor-corrector type discrete model in \cite{C-J-L-Z}.
\begin{equation}
\begin{cases} \label{B-4}
\displaystyle \hat X^i_{n} = {\bar X}_n^*+ e^{-\lambda h} (X^i_n - {\bar X}_n^*),\\
\displaystyle X^i_{n+1} =\hat X_n^i  -  \sum_{l=1}^{d} (\hat x^{i,l}_n - {\bar x}_n^{*,l}) \sigma \sqrt{h}Z_n^l e_l, \quad n\geq 0,~~ i = 1, \cdots, N.
\end{cases}
\end{equation}
We substitute $\eqref{B-4}_1$ into $\eqref{B-4}_2$ and use an addition-subtraction trick to see that 
\begin{equation*} \label{B-5}
X^i_{n+1} =X_n^i -(1-e^{-\lambda h}) (X^i_n - {\bar X}_n^*)  -  \sum_{l=1}^{d} (x^{i,l}_n - {\bar x}_n^{*,l}) e^{-\lambda h}\sigma \sqrt{h}Z_n^l e_l, \quad n\geq 0,~~ i = 1, \cdots, N.
\end{equation*}
If  we set 
\begin{equation} \label{B-5-1}
\gamma := 1-e^{-\lambda h} \quad \mbox{and} \quad \eta_n^l := e^{-\lambda h}\sigma \sqrt{h}Z_n^l,
 \end{equation}
then \eqref{B-4} reduces to the special case of \eqref{B-1} - \eqref{B-2} with $\zeta = e^{-\lambda h}\sigma \sqrt{h}$.

\vspace{0.5cm}

\noindent $\bullet$~Model C:  Consider one of discrete optimization model proposed in \cite{C-J-L-Z}:
\begin{equation}\label{B-6}
X^i_{n+1} ={\bar X}_n^*+  \sum_{l=1}^{d} (x^{i,l}_n - {\bar x}_n^{*,l})  \left[\exp\left(-\left(\lambda+\frac{1}{2}\sigma^2\right)h+\sigma\sqrt{h}Z_n^l \right)\right]e_l, \quad n\geq 0,~~ i = 1, \cdots, N,
\end{equation}
Again, the R.H.S. of \eqref{B-6} can be rewritten as 
\begin{equation*} \label{B-7}
X^i_{n+1} =X_n^i -(1-e^{-\lambda h}) (X^i_n - {\bar X}_n^*)  -  \sum_{l=1}^{d} (x^{i,l}_n - {\bar x}_n^{*,l})  e^{-\lambda h}\left[\exp\left(-\frac{1}{2}\sigma^2h+\sigma\sqrt{h}Z_n^l \right)-1\right]e_l.
\end{equation*}
We set 
\begin{equation} \label{B-8}
\gamma := 1-e^{-\lambda h} \quad \mbox{and} \quad \eta_n^l := e^{-\lambda h}\left[\exp\left(-\frac{1}{2}\sigma^2h+\sigma\sqrt{h}Z_n^l \right)-1\right].
\end{equation}
Then, we use the elementary facts \cite{C-S2}:
\[
X\sim\mbox{Lognormal}(\alpha,\beta^2)\quad\Rightarrow\quad \mathbb EX=e^{\alpha+\frac{\beta^2}{2}}\quad \mbox{and}\quad\mathbb EX^2=e^{2\alpha+2\beta^2}
\]
to see that \eqref{B-8}  satisfies moment relations \eqref{B-2} with $\zeta = e^{-\lambda h} \sqrt{e^{\sigma^2h}-1}$. 
\subsection{Global consensus} \label{sec:2.2} In this subsection, we show that the global consensus for \eqref{B-1} occurs asymptotically for all initial data under suitable conditions on $\gamma$ and $\zeta$. First, we recall the concepts of $L^p$ and almost sure global consensus in the following definition.
\begin{definition} \label{D2.1}
Let ${\mathcal X} = \{ {\mathcal X}_n := (x_n^1, \cdots, x_n^d) \}$ be a stochastic process, and let $(\Omega, {\mathcal F}, \bbp)$ be the underlying probability space.
\begin{enumerate}
\item
The configuration process ${\mathcal X}$ exhibits a global consensus in $L^p$ with $p \geq 1$, if the following zero $L^p$-convergence holds:
\[   \lim_{n \to +\infty} \bbe |X_n^i - X_n^j|^p = 0, \quad \forall~i, j = 1, \cdots, N.     \]
\item
The configuration process ${\mathcal X}$ exhibits a global consensus almost surely if for almost sure $\omega \in \Omega$ and $i, j = 1, \cdots, N$, the sample path $X_n^i(\omega) - X_n^j(\omega)$ tends to zero asymptotically.
\[   \bbp \Big \{ \lim_{n \to \infty}  |X_n^i - X_n^j | = 0 \Big \} = 1, \quad \forall~i, j = 1, \cdots, N.  \]
\end{enumerate}
\end{definition}
\begin{remark}
Recall that $L^2$-convergence implies $L^1$ convergence on a probability space.
\end{remark}
Now, we consider the relation for the process $X_n^i - X_n^j$. For this, we consider the discrete algorithms: for $i, j = 1, \cdots, N$, 
\begin{equation} \label{B-9}
\begin{cases}
\displaystyle X^i_{n+1} =X^i_n -\gamma(X^i_n - {\bar X}_n^*)  -\sum_{l=1}^{d} (x^{i,l}_n - {\bar x}_n^{*,l}) \eta_n^l e_l,  \\
\displaystyle X^j_{n+1} =X^j_n -\gamma(X^j_n - {\bar X}_n^*)  -\sum_{l=1}^{d} (x^{j,l}_n - {\bar x}_n^{*,l}) \eta_n^l e_l,
\end{cases}
\end{equation}
Then, it follows from \eqref{B-9} that $x^{i}_{n +1}  - x^{j}_{n +1}$ satisfies 
\begin{equation} \label{B-10}
x^{i,l}_{n +1}  - x^{j,l}_{n +1} = (1-\gamma - \eta^l_n  ) (x^{i,l}_n  - x_n^{j,l}). 
\end{equation}
In the following lemma, we provide several estimates for $x^{i,l}_{n}  - x^{j,l}_{n}$.  
\begin{lemma} \label{L2.1}
Let $\{{\mathcal X}_n \}$ be a solution process to \eqref{B-1}. Then the following estimates hold.
\begin{eqnarray*}
&& (i)~\mathbb E[X^{i}_n - X^{j}_n] =(1-\gamma)^n \mathbb E[X^{i}_0 - X^{j}_0]. \cr
&& (ii)~ \mathbb E|X^{i}_n - X^{j}_n|^2 =\big((1-\gamma)^2+\zeta^2\big)^n  \mathbb E|X^{i}_0 - X^{j}_0|^2.   \cr
&& (iii)~|x^{i,l}_n - x^{j,l}_n |^2\leq |x^{i,l}_0 - x^{j,l}_0|^2 e^{-n Y_n^l(\omega) }, \quad \mbox{a.s.}~\omega \in \Omega,  
\end{eqnarray*}
where $Y_n^l$ is a random variable satisfying 
\[ \lim_{n \to \infty} Y_n^l(\omega) =2\gamma-\gamma^2-\zeta^2 , \quad \mbox{a.s.}~\omega \in \Omega,\quad l=1,\cdots,d. \]
\end{lemma}
\begin{proof} The  estimates below for a special case (Model A) can be found in Theorem 3.4 in \cite{H-J-K} and proofs are almost similar. However, for self-containedness of this paper, we present their proofs. \newline

\noindent (i)~It follows from the recursive relation \eqref{B-10} that 
\begin{equation} \label{B-11}
x^{i,l}_n - x^{j,l}_n = (x^{i,l}_0 - x^{j,l}_0)\prod_{m = 0}^{n-1} (1-\gamma - \eta^l_m  ) ,\qquad l=1,\cdots,d.
\end{equation}
Now we take expectation on both sides of \eqref{B-11} using the independence of $\eta_m^l$ and $x^{i,l}_0 - x^{j,l}_0$ to get
\[ \mathbb E[x^{i,l}_n - x^{j,l}_n] =(1-\gamma)^n  \mathbb E[ x^{i,l}_0 - x^{j,l}_0 ], \quad \mbox{i.e.,} \quad \mathbb E[X^{i}_n - X^{j}_n]  =(1-\gamma)^n \mathbb E[X^{i}_0 - X^{j}_0].
\]

\vspace{0.5cm}

\noindent (ii)~We take the square of \eqref{B-11} to see
\begin{equation} \label{B-12}
|x^{i,l}_n - x^{j,l}_n |^2= |x^{i,l}_0 - x^{j,l}_0|^2\prod_{m = 0}^{n-1} |1-\gamma - \eta^l_m  |^2 ,\qquad l=1,\cdots,d.
\end{equation}
Then, \eqref{A-3}, \eqref{B-12}, independence of $\eta_m^l$ and $x^{i,l}_0 - x^{j,l}_0$ yield
%\begin{equation} \label{D-0-12}
%\bbe |X^i_n - X^j_n|^2  =  \Pi_{\ell = 0}^{n-1} \bbe[(1- \Delta_\ell )^2] \times \bbe[|X^i_0 - X^j_0|^2].  
%\end{equation}
%On the other hand, since $\{\Delta_\ell \}$ are i.i.d. and for each $\ell = 0, \cdots, n-1$,  one has 
%\begin{align}
%\begin{aligned}  \label{D-0-13}
%\bbe[(1- \Delta_\ell )^2] &= 1-{\mathbb E} [2\Delta_\ell - \Delta_\ell^2 ] \\
%&=1- {\mathbb E} \Big(2 \lambda h - 2\sigma \sqrt{h} Z_\ell -  (\lambda h)^2 + 2\sigma \lambda h^{\frac{3}{2}} Z_\ell - \sigma^2 h Z_\ell^2 \Big) \\
%&=1- 2\lambda h + \lambda^2 h^2 + \sigma^2 h = 1 -h m(\lambda,h, \sigma) \geq 0.
%\end{aligned}
%\end{align}
%Now, we combine \eqref{D-0-12} and \eqref{D-0-13} to get 
\[ \mathbb E|x^{i,l}_n - x^{j,l}_n|^2 =\big((1-\gamma)^2+\zeta^2\big)^n  \mathbb E|x^{i,l}_0 - x^{j,l}_0|^2, \]
i.e.,
\[ \mathbb E|X^{i}_n - X^{j}_n|^2 =\big((1-\gamma)^2+\zeta^2\big)^n  \mathbb E|X^{i}_0 - X^{j}_0|^2. \]

%In order to get the decay of estimate of the R.H.S. of  the above relation, we need to check the following relation:
%\[ 0 < 1- h m < 1, \quad \mbox{or} \quad 0 < 1- 2\lambda h + \lambda^2 h^2 + \sigma^2 h < 1. \]
%Note that 
%\[ 1- 2\lambda h + \lambda^2 h^2 + \sigma^2 h < 1 \quad \Longleftrightarrow \quad  h < \frac{2\lambda-\sigma^2}{\lambda}.  \]
%and the other side inequality holds for all $h > 0$:
%\[  \lambda^2 h^2  + (\sigma^2 - 2 \lambda) h + 1 =(1-\lambda h)^2+\sigma^2 h>0.  \]

\vspace{0.5cm}

\noindent (iii)~It follows from the inequality $e^x \geq 1 + x,~~ x \in \bbr$ that
\begin{equation} \label{B-13}
(1 -\gamma - \eta^l_m)^2 \leq e^{(1 -\gamma - \eta^l_m)^2-1} =e^{-(\gamma + \eta^l_m)(2 -\gamma - \eta^l_m)}.
\end{equation}
Then, we use \eqref{B-12} and \eqref{B-13} to obtain
\begin{align*}
\begin{aligned} \label{B-14}
|x^{i,l}_n - x^{j,l}_n |^2&\leq |x^{i,l}_0 - x^{j,l}_0|^2\prod_{m = 0}^{n-1}e^{-(\gamma + \eta^l_m)(2 -\gamma - \eta^l_m)} \\
&= |x^{i,l}_0 - x^{j,l}_0|^2\exp \Big[ -n \times \frac{1}{n}\sum_{m = 0}^{n-1} (\gamma + \eta^l_m)(2 -\gamma - \eta^l_m) \Big] .
\end{aligned}
\end{align*}
We set 
\[  Y_n := \frac{1}{n}\sum_{m = 0}^{n-1} (\gamma + \eta^l_m)(2 -\gamma - \eta^l_m). \]
Then, we  use the strong law of large numbers to see
\[  Y_n \quad  \longrightarrow \quad {\mathbb E} \Big[  (\gamma + \eta^l_m)(2 -\gamma - \eta^l_m)  \Big] =2\gamma-\gamma^2-\zeta^2 \quad \mbox{a.s.~~as 
$n \to \infty$.}  \]
\end{proof}
As a direct application of Lemma \ref{L2.1}, we have the following global consensus estimates.
\begin{theorem} \label{T2.1}
Let $\{{\mathcal X}_n \}$ be a solution process to \eqref{B-1}.  Then, the following three global consensus results hold.
\begin{enumerate}
\item
Suppose that system parameters satisfy 
\[ |\gamma - 1| < 1 \quad \mbox{and} \quad 0\leq \zeta\leq \infty . \]
Then, $ \mathbb E[X^{i}_n - X^{j}_n]$ tends to zero asymptotically:
\[ \lim_{ n \to \infty} \mathbb E[X^{i}_n - X^{j}_n]  = 0, \quad \forall~i, j = 1, \cdots, N. \]
\item 
Suppose that system parameters $\gamma$ and $\zeta$ satisfy 
\[ (\gamma-1)^2 + \zeta^2 < 1    \]
then, $L^2$ and almost-sure global consensus emerge asymptotically: for a.s. $\omega \in \Omega$,
\[ \lim_{n \to \infty} \mathbb E|X^{i}_n - X^{j}_n|^2 = 0, \quad |x^{i,l}_n - x^{j,l}_n |^2\leq |x^{i,l}_0 - x^{j,l}_0|^2 e^{-n Y_n^l(\omega) }, \quad i, j = 1, \cdots, N,~~l = 1, \cdots, d,
\]
where $Y_n^l$ is a random variable satisfying 
\[ \lim_{n \to \infty} Y_n^l(\omega) = 1 - (\gamma - 1)^2 - \zeta^2 > 0, \quad \mbox{a.s.}~\omega \in \Omega,\quad l=1,\cdots,d. \]
\end{enumerate}
\end{theorem}
\begin{proof} Since the proofs follow from Lemma \ref{L2.1} directly, we omit them here. 
\end{proof}
As a corollary of Theorem \ref{T2.1}, one has the following global consensus for Model A - Model C discussed in previous subsection.  
\begin{corollary} \label{C2.1}
The following assertions hold.
\begin{enumerate}
\item
Suppose that system parameters satisfy
\begin{equation*}  \label{stab-A}
\lambda > \frac{\sigma^2}{2}, \quad  0 < h < \frac{2\lambda-\sigma^2}{\lambda^2},
\end{equation*}
then, Model A admits $L^2$ and almost sure global consensus. 

\vspace{0.2cm}

\item
Suppose that system parameters satisfy
\begin{equation*}  \label{stab-B}
(1+\sigma^2h)e^{-2\lambda h}<1,
\end{equation*}
then, Model B admits $L^2$ and almost sure global consensus. .

\vspace{0.2cm}

\item
Suppose that system parameters satisfy 
\begin{equation} \label{stab-C} 
\lambda> \frac{\sigma^2}{2}, 
\end{equation}
then, Model C admits $L^2$ and almost sure global consensus, for any $h>0$.
\end{enumerate}
\end{corollary}
\begin{proof}
(i) For Model A, we use relations $\gamma = \lambda h,~~\zeta = \sigma \sqrt{h}$ in  \eqref{B-3-1} to see
\[ 2\gamma-\gamma^2-\zeta^2 = 2\lambda h -(\lambda h)^2-\sigma^2 h>0,\quad\mbox{for}\quad 0<h<\frac{2\lambda-\sigma^2}{\lambda^2}. \]
Then, we use (3) in Theorem \ref{T2.1} to derive the desired estimate. \newline

\noindent (2)~For Model B, we use relations $\gamma =  1 - e^{-\lambda h},~~\zeta = \zeta = e^{-\lambda h}\sigma \sqrt{h}$ in \eqref{B-5-1} to see
\[
2\gamma-\gamma^2-\zeta^2 = 2(1-e^{-\lambda h}) -(1-e^{-\lambda h})^2-\sigma^2 he^{-2\lambda h}=1-(1+\sigma^2h)e^{-2\lambda h}>0.
\]
(3) For Model C, we use $\gamma = 1 - e^{-\lambda h},~~\zeta = e^{-\lambda h} \sqrt{e^{\sigma^2 h} - 1}$ in \eqref{B-8} to get
\[
2\gamma-\gamma^2-\zeta^2 = 2(1-e^{-\lambda h}) -(1-e^{-\lambda h})^2-(e^{\sigma^2h}-1) e^{-2\lambda h}=1-e^{(\sigma^2-2\lambda) h}>0\quad\Leftrightarrow\quad (2\lambda-\sigma^2)h>0.
\]
\end{proof}

\begin{remark} 
The above corollary provides the stability conditions for three algorithms. For Model C, the algorithm is {\it unconditionally stable} provided \eqref{stab-C} holds, namely one can choose arbitrary $h>0$.
Model B is also unconditionally stable provided $\lambda\geq\frac{\sigma^2}{2}$ is satisfied, because
\[ e^{2\lambda h}>1+ 2\lambda h \geq 1+ \sigma^2 h ,\quad\mbox{for}\quad h>0. \]
Model A is conditionally stable. 
\end{remark}

\section{Convergence analysis and error estimates}  \label{sec:3}
\setcounter{equation}{0}
In this section, we provide a convergence analysis with error estimates for the discrete CBO algorithm.  In previous section, we studied sufficient conditions leading to the global consensus which does not mean $X_n^i$ tends to a common fixed state $X_\infty$. Thus, two main issues to be covered in this section are two-fold.
\begin{itemize}
\item
(Q1):~What is a sufficient framework leading to the common asymptotic state:
\begin{equation*} \label{C-1}
 X_n^i (\beta) \to X_\infty(\beta), \quad \mbox{as $n \to \infty$~for all $i = 1, \cdots, N$?} 
\end{equation*} 
\item
(Q2):~If the above question is resolved, then how close is the asymptotic state $X_\infty$ to the global minimum $X_m$ of $L$ if the latter exists? %For a unique global minimum $L_m := \min_{x \in \bbr^d} L(x)$, do we have
%\begin{equation*} \label{C-2}
 %\lim_{\beta \to \infty} L(X_\infty(\beta)) = L_m?
%\end{equation*} 
\end{itemize}

\subsection{Emergence of common consensus state} \label{sec:3.1}
For the emergence of common consensus state, we first introduce an ensemble average:
\[
{\bar X}_n:=\frac{1}{N}\sum_{i=1}^N X_n^i = ({\bar x}_n^1, \cdots, {\bar x}_n^d).
\]
Before we present our second main result, we present two elementary lemmas to be crucially used in the proof of convergence analysis.
\begin{lemma}\label{L3.1}
Let $\{X_n^i\}_{1\leq i\leq N}$ be a solution to \eqref{B-1}. Then, the following estimates hold almost surely. 
\begin{eqnarray*}
&& (i)~|X_n^i-\bar X_n|^2=\sum_{l=1}^d (x_0^{i,l} -{\bar x}^l_0)^2\prod_{m=0}^{n-1}\left(1 - \gamma-\eta_m^l\right)^2.\\
&& (ii)~|\bar X_n-\bar X_n^*|^2\leq \max_{1\leq i\leq N}|X_n^i-\bar X_n|^2.\\
&& (iii)~  \frac{1}{N}\sum_{i=1}^d |X_n^i-\bar X_n^*|^2\leq 2 \sum_{l=1}^d \left(\max_{1\leq i\leq N} (x_0^{i,l} -{\bar x}^l_0)^2\right)\prod_{m=0}^{n-1}\left(1 - \gamma-\eta_m^l\right)^2.
\end{eqnarray*}
\end{lemma}
\begin{proof}
(i)~It follows from \eqref{B-1} that 
\begin{equation}\label{C-3}
{\bar X}_{n+1}-{\bar X}_{n}=-\gamma({\bar X}_n-\bar X_n^*)-\sum_{l=1}^d( {\bar x}^l_n- {\bar x}_n^{*,l}) \eta_n^l e_l.
\end{equation}
We subtract \eqref{C-3} from \eqref{B-1} to obtain
\begin{equation} \label{C-4}
(X_{n+1}^i- {\bar X}_{n+1})-(X_{n}^i- {\bar X}_{n})=-\gamma(X_n^i- {\bar X}_n)-\sum_{l=1}^d( x_n^{i,l}- {\bar x}^{l}_n) \eta_n^l  e_l.
\end{equation}
The $l$-th component of \eqref{C-4} implies
\[
x_n^{i,l} -{\bar x}^{l}_n = (x_0^{i,l} -{\bar x}^l_0) \prod_{m=0}^{n-1}\left(1 - \gamma-\eta_m^l\right).
\]
This yields
\[
|X_n^i-\bar X_n|^2=\sum_{l=1}^d (x_0^{i,l} -{\bar x}^l_0)^2 \prod_{m=0}^{n-1}\left(1 - \gamma-\eta_m^l\right)^2.
\]
%We take maximum over $i$ and then take expectation on both side to see
%\begin{align*}
%\begin{aligned}
%&\mathbb E \Big[ \max_{1\leq i\leq N}|X_t^i- {\bar X}_t|^2 \Big] \\
%& \hspace{0.5cm} \leq  \mathbb E \Big[  \max_{1\leq i\leq N}|X_0^i- {\bar X}_0|^2\Big] \cdot \exp\left[ -\Big(2\lambda+\sigma^2\Big)t\right ] \cdot \mathbb E \Big [ \max_{1\leq k\leq d}\exp\left(2\sigma W_t^k\right ) \Big] \\
%& \hspace{0.5cm} \leq  \mathbb E \Big[ \max_{1\leq i\leq N}|X_0^i-{\bar X}_0|^2\Big ] \cdot \exp\left[ -\Big(2\lambda+\sigma^2\Big)t\right]  \cdot \sum_{1\leq k\leq d}\mathbb E  \Big[ \exp\Big(2\sigma W_t^k \Big) \Big] \\
%& \hspace{0.5cm} \leq d  \mathbb E \Big[ \max_{1\leq i\leq N}|X_0^i- {\bar X}_0|^2\Big ]  \cdot \exp\left[ -\Big(2\lambda+\sigma^2\Big)t\right] {\color{blue}{\exp\left(2\sigma^2 t\right)}}\\
%& \hspace{0.5cm} = de^{-(2\lambda-\sigma^2)t} \mathbb E \Big[ \max_{1\leq i\leq N}|X_0^i-{\bar X}_0|^2\Big].
%\end{aligned}
%\end{align*}

\vspace{0.5cm}

\noindent (ii)~We apply the triangle inequality and the Cauchy-Schwarz inequality:
\begin{align*}
\begin{aligned}
|{\bar X}_n-\bar X_n^*|^2&=\left|\frac{\sum_{k=1}^Ne^{-\beta L(X_n^k)}( {\bar X}_n-X_n^k)}{\sum_{k=1}^Ne^{-\beta L(X_n^k)}}\right|^2\leq \left[\frac{\sum_{k=1}^Ne^{-\beta L(X_n^k)}| {\bar X}_n-X_n^k|}{\sum_{k=1}^Ne^{-\beta L(X_n^k)}}\right]^2\\
&\leq \frac{\sum_{k=1}^N e^{-\beta L(X_n^k)}| {\bar X}_n-X_n^k|^2}{\sum_{k=1}^Ne^{-\beta L(X_n^k)}}\leq \max_{1\leq k \leq N}| {\bar X}_n-X_n^k |^2.
\end{aligned}
\end{align*}

\vspace{0.5cm}

\noindent (iii)~ Note that
\begin{align*}
\begin{aligned}
\frac{1}{N}\sum_{i=1}^d |X_n^i-\bar X_n^*|^2&=\frac{1}{N}\sum_{i=1}^d\left(|X_n^i-\bar X_n|^2+2(X_n^i-\bar X_n)\cdot(\bar X_n-\bar X_n^*)+|\bar X_n-\bar X_n^*|^2 \right)\\
&=\frac{1}{N}\sum_{i=1}^d\left( |X_n^i-\bar X_n|^2 +|\bar X_n-\bar X_n^*|^2 \right)\leq 2\max_{1\leq i\leq N} |X_n^i-\bar X_n|^2\\
&=2\max_{1\leq i\leq N} \left(\sum_{l=1}^d (x_0^{i,l} -{\bar x}^l_0)^2\prod_{m=0}^{n-1}\left(1 - \gamma-\eta_m^l\right)^2\right)\\
&\leq 2 \sum_{l=1}^d \left(\max_{1\leq i\leq N} (x_0^{i,l} -{\bar x}^l_0)^2\right)\prod_{m=0}^{n-1}\left(1 - \gamma-\eta_m^l\right)^2,
\end{aligned}
\end{align*}
where we used the results from (i) and (ii).
%(iii)~Note that
%\begin{align}\label{D-6-1}
%\begin{aligned}
%\frac{1}{N}\sum_{i=1}^N |X_t^i-\bar X_t^*|^2&=\frac{1}{N}\sum_{i=1}^N \left(|X_t^i-\bar X_t|^2+2(X_t^i-\bar X_t)\cdot(\bar X_t-\bar X_t^*)+|\bar X_t-\bar X_t^*|^2\right)\\
%&=\frac{1}{N}\sum_{i=1}^N  |X_t^i-\bar X_t|^2 +|\bar X_t-\bar X_t^*|^2 \leq 2\max_{1\leq i\leq N} |X_t^i-\bar X_t|^2\\
%&=2\max_{1\leq i\leq N} \left(\sum_{l=1}^d (x_0^{i,l} -{\bar x}^l_0)^2\exp\left[ -\Big(2\lambda+\sigma^2\Big)t+2\sigma W_t^l\right]\right)\\
%&\leq 2 \sum_{l=1}^d \left(\max_{1\leq i\leq N} (x_0^{i,l} -{\bar x}^l_0)^2\right)\exp\left[ -\Big(2\lambda+\sigma^2\Big)t+2\sigma W_t^l\right],
%\end{aligned}
%\end{align}
%where we used the inequalities from (i) and (ii).
\end{proof}
\begin{lemma}\label{L3.2}
Let $\{X_n^i\}_{1\leq i\leq N}$ be a solution to \eqref{B-1}. Then, one has
\[ \mathbb E|X_n^i-\bar X_n^*|^2\leq 2((1-\gamma)^2 + \zeta^2)^n \sum_{l=1}^d \left(\mathbb E\max_{1\leq i\leq N} (x_0^{i,l} -{\bar x}^l_0)^2\right). \]
\end{lemma}
\begin{proof}
We take expectation on both sides of the inequality in Lemma \ref{L3.1} (iii) to get
\begin{align*}
\begin{aligned}
 \mathbb E|X_n^i-\bar X_n^*|^2&\leq 2 \sum_{l=1}^d \left(\mathbb E\max_{1\leq i\leq N} (x_0^{i,l} -{\bar x}^l_0)^2\right)\prod_{m=0}^{n-1}\mathbb E\left(1 - \gamma-\eta_m^l\right)^2\\
&\leq 2 \Big((1-\gamma)^2 + \zeta^2 \Big)^n \sum_{l=1}^d \left(\mathbb E\max_{1\leq i\leq N} (x_0^{i,l} -{\bar x}^l_0)^2\right).
\end{aligned}
\end{align*}
\end{proof}
Next, we provide our second main result. 
\begin{theorem}\label{T3.1}
Suppose that system parameters satisfy 
\[   (1 -\gamma)^2 + \zeta^2 < 1, \]
and let $\{X_n^i\}_{1\leq i\leq N}$ be a solution to \eqref{B-1}. Then, there exists a common constant state $X_\infty=(x_\infty^1,\cdots,x_\infty^d)$ such that
\[  \lim\limits_{n\to\infty}  X_n^i=X_\infty~\mbox{a.s.},~1\leq i\leq N. \]
\end{theorem}
\begin{proof}
Note that summing the equation $\eqref{B-1}$ over $n$ yields the following relation: for $i=1,\cdots,N$ and $l=1,\cdots,d$, 
\begin{equation*}\label{C-5}
x_n^{i,l}=x_0^{i,l}-\gamma \sum_{m=0}^{n-1} (x_m^{i,l} - \bar x_m^{*,l})  -  \sum_{m=0}^{n-1} (x_m^{i,l} - \bar x_m^{*,l}) \eta_m^l=:x_0^{i,l}- {\mathcal I}_{31} -  {\mathcal I}_{32}.
\end{equation*}
Next, we show the a.s. convergence of ${\mathcal I}_{31}$ and ${\mathcal I}_{32}$ separately.\newline

\noindent $\bullet$~(Estimate of ${\mathcal I}_{31}$):~ By Lemma \ref{L3.1} (iii), we have 
\begin{align*}
\begin{aligned}
|x_m^{i,l} - \bar x_m^{*,l}| &\leq  \sum_{i=1}^d|X_m^i-\bar X_m^*| \\
&\leq \sqrt{2N\sum_{l=1}^d \left(\max_{1\leq i\leq N} (x_0^{i,l} -{\bar x}^l_0)^2\right)\prod_{p=0}^{m-1}\left(1 - \gamma- \eta_p^l\right)^2}\\
&\leq \sqrt{2N\sum_{l=1}^d \left(\max_{1\leq i\leq N} (x_0^{i,l} -{\bar x}^l_0)^2\right)\prod_{p=0}^{m-1}\exp\left((1 - \gamma- \eta_p^l)^2-1\right)}\\
&= \sqrt{2N\sum_{l=1}^d \left(\max_{1\leq i\leq N} (x_0^{i,l} -{\bar x}^l_0)^2\right)\exp\left(\sum_{p=0}^{m-1}(-2\gamma+\gamma^2-2(1-\gamma)\eta_p^l+(\eta_p^l)^2)\right)}.
\end{aligned}
\end{align*}
On the other hand, by the strong law of large numbers, one has
\begin{align*}
&\lim_{m\to\infty} \frac{1}{m}\sum_{p=0}^{m-1} \Big(-2\gamma+\gamma^2-2(1-\gamma)\eta_p^l+(\eta_p^l)^2 \Big ) \\
& \hspace{1cm} =\mathbb E\Big[ -2\gamma+\gamma^2-2(1-\gamma)\eta_p^l+(\eta_p^l)^2 \Big] =-(2\gamma  - \gamma^2 -\zeta^2 ),\quad \mbox{a.s.}
\end{align*}
This yields that there exist positive random functions $C_i = C_i(\omega),~i= 1,2$ such that 
\[ |x_m^{i,l} - \bar x_m^{*,l}|\leq C_1e^{-C_2m},  \quad \mbox{a.s.}~\omega \in \Omega, \] 
where $C_1$ and $C_2$ are positive constants. We set 
\[ {\mathcal J}_{31} := {\mathcal I}_{31} -\gamma \sum_{m=0}^{n-1} C_1e^{-C_2m}  = \gamma \sum_{m=0}^{n-1} (\underbrace{x_m^{i,l} - \bar x_m^{*,l}- C_1e^{-C_2m}}_{\leq 0}). \]
Since the summand is nonpositive a.s., $ {\mathcal J}_{31}$ is non-increasing in $n$ a.s.  \newline

\noindent On the other hand, note that 
\begin{align*}
\begin{aligned}
 {\mathcal J}_{31}  &= {\mathcal I}_{31} +\gamma \sum_{m=0}^{n-1} C_1e^{-C_2m}-2\gamma \sum_{m=0}^{n-1} C_1e^{-C_2m}\\
&= \gamma \sum_{m=0}^{n-1} (\underbrace{x_m^{i,l} - \bar x_m^{*,l}+ C_1e^{-C_2m}}_{\geq 0}) -2\gamma C_1\frac{1-e^{-C_2n}}{1-e^{-C_2}} \geq -\frac{2\gamma C_1}{1-e^{-C_2}}.
\end{aligned}
\end{align*}
Since ${\mathcal J}_{31}$ is monotone decreasing and bounded below along sample paths, one has
\[  \exists~\alpha = \lim_{n\to \infty} {\mathcal J}_{31}(n) = \lim_{n\to \infty}  \Big( {\mathcal I}_{31}  -\gamma \sum_{m=0}^{n-1} C_1e^{-C_2m}  \Big), \quad \mbox{a.s.} \]
This implies
\[ \lim_{n\to\infty} {\mathcal I}_{31}=\alpha+\frac{C_1\gamma}{1-e^{-C_2}}, \quad \mbox{a.s.} \]

\noindent $\bullet$~(Estimate of ${\mathcal I}_{32}$):~Note that ${\mathcal I}_{32}$ is martingale and its $L^2(\Omega)$-norm is uniformly bounded in $n$:
\begin{align*}
\begin{aligned}
&\mathbb E \left[\sum_{m=0}^{n-1} (x_m^{i,l} - \bar x_m^{*,l}) \eta_m^l\right]^2 =\zeta^2\sum_{m=0}^{n-1}\mathbb E (x_m^{i,l} - \bar x_m^{*,l})^2 \\
& \hspace{1cm} \leq 2N\zeta^2\left(\sum_{m=0}^{n-1}((1-\gamma)^2 + \zeta^2)^m \right) \sum_{l=1}^d \left(\mathbb E\max_{1\leq i\leq N} (x_0^{i,l} -{\bar x}^l_0)^2\right) \\
& \hspace{1cm} \leq  \frac{2N\zeta^2 }{2\gamma  - \gamma^2 -\zeta^2}  \sum_{l=1}^d \left(\mathbb E\max_{1\leq i\leq N} (x_0^{i,l} -{\bar x}^l_0)^2\right).
\end{aligned}
\end{align*}
In the first inequality we used (iii). Hence $\lim\limits_{n\to\infty} {\mathcal I}_{32}$ exists a.s. Now we have shown that for each $i=1,\cdots, N$, there exists some random variable $X^{i}_{\infty}$ such that 
\[ \lim\limits_{n\to\infty}X_n^i=X^{i}_{\infty} \quad \mbox{a.s.} \]
Recall that by Theorem \ref{T2.1}, for any $1\leq i,j\leq N$,  
\[ \lim\limits_{n\to\infty}|X_n^i-X_n^j|=0, \quad \mbox{a.s.} \]
Hence, there exists $X_\infty$ such that 
\[ X^{i}_{\infty}=X^{j}_{\infty}=:X_\infty \quad \mbox{a.s.} \]
\end{proof}

\subsection{Error estimate} \label{sec:3.2}
In this subsection, we study an error analysis of \eqref{B-1} toward the global minumum. Below, we present a sufficient framework $({\mathcal A}1) - ({\mathcal A}3)$ for error analysis in terms of the objective function $L$, global minimum point $X_*$ and reference 
random variable $X_{in}$ as follows: \newline
\begin{itemize}
\item
$({\mathcal A}1)$:~Let $L = L(x)$ be a $C^2$-objective function  satisfying the following relations:
\begin{equation*} \label{C-5}
 L_m:=\min_{x \in \bbr^d} L(x) >0 \quad \mbox{and} \quad C_L:=\sup_{x\in\bbr^d}\|\nabla^2L(x)\|_2 <\infty,
\end{equation*} 
where $\|\cdot\|_2$ denotes the spectral norm. 

\vspace{0.2cm}

\item
$({\mathcal A}2)$:~Let $X_*$ be the unique global minimum point of $L$ in $\bbr^d$ satisfying the local convexity relation:
\begin{equation*} \label{C-5-0}
\operatorname{det}\left(\nabla^2 L(X_*)\right)>0.
\end{equation*}

\vspace{0.2cm}

\item 
$({\mathcal A}3)$:~Let $X_{in}$ be a reference random variable with a law which is absolutely continuous with respect to the Lebesgue measure, and let $f$ be the probability density function
of $X_{in}$ satisfying the following conditions:
\begin{equation*} \label{C-5-1}
\mbox{$f$ is compactly supported, continuous at $X_*$, and $f(X_*)>0$}. 
\end{equation*}
\end{itemize}
In the next theorem, we study how close is the common consensus state $X_\infty$ to the global minimum $X_*$ in suitable sense.
Now, we are ready to provide an error analysis of the discrete CBO algorithm, which is analogous to Theorem 4.1 in \cite{H-J-K} for the continuous case. 
\begin{theorem} \label{T3.2}
Suppose that the framework $({\mathcal  A}1) - ({\mathcal  A}3)$ holds, and system parameters $\beta, \gamma, \zeta$ and the initial data $\{X_0^i \}$ satisfy 
\begin{align}
\begin{aligned} \label{C-5-2}
& \beta > 0, \quad  (\gamma-1)^2 + \zeta^2 < 1, \quad X_0^i: i,i.d, \quad X_0^i \sim X_{in}, \\
& (1-\varepsilon)\mathbb E \Big[ e^{-\beta L(X_{in})} \Big ] \\
& \hspace{0.5cm} \geq \frac{2C_L \sqrt{ \big(1+ (1-\gamma)^2+\zeta^2\big) \big( \gamma^2+\zeta^2\big)} \beta  e^{-\beta L_m}}{1-e^{-[ 1-(\gamma - 1)^2 - \zeta^2]}} \sum_{l=1}^d \left(\mathbb E\max_{1\leq i\leq N} (x_0^{i,l} -{\bar x}^l_0)^2\right),
\end{aligned}
\end{align}
for some $0<\varepsilon<1$. Then for a solution $\{X_n^i\}_{1\leq i\leq N}$ to \eqref{A-1}, one has the following error estimate:
\begin{equation} \label{C-5-3}
\Big| \essinf\limits_{\omega\in\Omega} L(X_\infty) - L(X_*) \Big| \leq \frac{d}{2}\frac{\log\beta}{\beta}+ E(\beta),
\end{equation}
for some function $E(\beta)= {\mathcal O}\left(\frac{1}{\beta} \right)$.
\end{theorem}
\begin{remark}\label{R3.1}
1. Note that the third condition $\eqref{C-5-2}_2$:
\begin{equation} \label{C-5-4}
(1-\varepsilon)\mathbb E \Big[ e^{-\beta L(X_{in})} \Big ] \geq \frac{2C_L \sqrt{ \big(1+ (1-\gamma)^2+\zeta^2\big) \big( \gamma^2+\zeta^2\big)} \beta  e^{-\beta L_m}}{1-e^{-[ 1-(\gamma - 1)^2 - \zeta^2]}} \sum_{l=1}^d \left(\mathbb E\max_{1\leq i\leq N} (x_0^{i,l} -{\bar x}^l_0)^2\right),
\end{equation}
does hold only for some intermediate $\beta$ for given initial data satisfying
\[ \sum_{l=1}^d \left(\mathbb E\max_{1\leq i\leq N} (x_0^{i,l} -{\bar x}^l_0)^2\right) > 0. \]
This can be checked as follows. We multiply $e^{\beta L_m}$ on the both sides of \eqref{C-5-4} to get
\[ (1-\varepsilon)\mathbb E \Big[ e^{\beta(L_m - L(X_{in}))} \Big ] \geq \frac{2C_L \sqrt{ \big(1+ (1-\gamma)^2+\zeta^2\big) \big( \gamma^2+\zeta^2\big)} \beta }{1-e^{-(2\gamma  - \gamma^2 -\zeta^2)}} \sum_{l=1}^d \left(\mathbb E\max_{1\leq i\leq N} (x_0^{i,l} -{\bar x}^l_0)^2\right). \]
By letting $\beta \to \infty$ in the above relation, we can see that L.H.S. is less than or equal to $1-\varepsilon$, but R.H.S. goes to $+\infty$. Hence, the estimate \eqref{C-5-3} is only an error estimate for the discrete CBO algorithm for fixed $\beta$ (indeed numerically one chooses a fixed $\beta$ for computation), not viewed as a convergence for $\beta \to \infty$ (which would impose restrictive  constratin on the initial data from $\eqref{C-5-2}_2$. In Appendix A, we provide an alternative error analysis for \eqref{A-1} without employing Laplace's principle. \newline

\noindent 2. In the sequel, we check how the conditions in \eqref{C-5-2} can be reinterpreted for Model A - Model C in Section \ref{sec:2.1}: \newline

\noindent $\diamond$~{\bf Model A}:~We use $(\lambda, \zeta) = (\lambda h, \sigma \sqrt{h})$ to rewrite \eqref{C-5-2} as 
\begin{align*}
\begin{aligned} 
& \beta > 0, \quad \lambda > \frac{\sigma^2}{2}, \quad 0 < h  < \frac{2\lambda - \sigma^2}{\lambda^2}, \\
& (1-\varepsilon)\mathbb E \Big[ e^{-\beta L(X_{in})} \Big ] \geq \frac{2C_L \sqrt{h(\sigma^2 + \lambda^2 h)[ 2 + h (\sigma^2 + \lambda^2 h - 2\lambda)]} \beta  e^{-\beta L_m}}{1-e^{-\lambda(2h  - \lambda h^2 -\sigma^2)  }} \\
& \hspace{3.5cm} \times \sum_{l=1}^d \left(\mathbb E\max_{1\leq i\leq N} (x_0^{i,l} -{\bar x}^l_0)^2\right).
\end{aligned}
\end{align*}
\noindent $\diamond$~{\bf Model B}:~We use $(\lambda, \zeta) = (1 - e^{-\lambda h}, e^{-\lambda h} \sigma \sqrt{h})$ to rewrite \eqref{C-5-2} as 
\begin{align*}
\begin{aligned}
& \beta > 0, \quad h > 0, \quad (1+\sigma^2h)e^{-2\lambda h}<1, \\
& (1-\varepsilon)\mathbb E \Big[ e^{-\beta L(X_{in})} \Big ] \geq \frac{2C_L \sqrt{[ 1 + e^{-\lambda h} (-2 + e^{-\lambda h} (1 + \sigma^2 h))][ 1 + e^{-2\lambda h} (1 + \sigma^2 h)] } \beta  e^{-\beta L_m}}{1-e^{-[1 - (1 + \sigma^2 h) e^{-2\lambda h}]}  } \\
& \hspace{3.5cm} \times \sum_{l=1}^d \left(\mathbb E\max_{1\leq i\leq N} (x_0^{i,l} -{\bar x}^l_0)^2\right).
\end{aligned}
\end{align*}
\noindent $\diamond$~{\bf Model C}:~We use $(\lambda, \zeta) =(1 - e^{-\lambda h}, e^{-\lambda h} \sqrt{e^{\sigma^2 h} - 1})$ to rewrite \eqref{C-5-2} as 
\begin{align*}
\begin{aligned}
& \beta > 0, \quad h > 0,  \quad \lambda > \frac{\sigma^2}{2}, \\
& (1-\varepsilon)\mathbb E \Big[ e^{-\beta L(X_{in})} \Big ] \geq 
\frac{2C_L \sqrt{ (1 + e^{h(\sigma^2 - 2h)}) ( 1 + e^{-\lambda h} (e^{-\lambda h} + e^{\sigma^2 h} - 3) )} \beta  e^{-\beta L_m}}{1-e^{-(1 - e^{(\sigma^2 - 2\lambda) h })}} \\
& \hspace{3.5cm} \times \sum_{l=1}^d \left(\mathbb E\max_{1\leq i\leq N} (x_0^{i,l} -{\bar x}^l_0)^2\right).
\end{aligned}
\end{align*}
\end{remark}

\vspace{0.5cm}

Before we present a proof of Theorem \ref{T3.2}, we first briefly review Laplace's principle with a convergence rate. \newline

\noindent $\bullet$~{\bf (Laplace's principle with a rate for $\beta \gg 1$)}:~Note that under some suitable conditions on $L$ and the law of $X_{in}$, we can apply Varadhan's lemma (cf. \cite{D-Z}) to get
\begin{equation} \label{New-1}
\lim_{\beta\to\infty}\left(-\frac{1}{\beta}\log \mathbb E e^{-\beta L(X_{in})}\right)=\essinf\limits_{\omega\in\Omega} L(X_{in}).
\end{equation}
However, as far as the authors know, standard proofs for Varadhan's lemma do not yield a convergence rate. Hence, we try a different approach for a possible convergence rate in \eqref{New-1}. \newline
%\begin{lemma}\label{L3.3}\cite{D-Z2}
%Let $A$ be an invertible $d\times d$ matrix, and let $u$ and $v$ be two $d$-dimensional column vectors. Then
%\[
%\operatorname{det}(A+uv^\top)=(1+v^\top A^{-1}u)\operatorname{det}A
%\]
%\end{lemma}
%\begin{proof}
%See Lemma 1.1 in \cite{D-Z2}.
%\end{proof}

Recall that the set $D \in \bbr^d$ is called a {\it d-dimensional closed domain} if it is a bounded finitely connected open set (in $\bbr^d$) plus its boundary, where the boundary is a Euclidean $(d-1)$-surface. One simple example of such $D$ is a closed ball in $\bbr^d$. In the sequel, let $D$ be a $d$-dimensional closed domain. 
\begin{lemma}\label{L3.3}\cite{Hsu}
Suppose that two functions $\phi:D\to\bbr$ and $h:D\to\bbr$ satisfy the following conditions:
\begin{enumerate}
\item $h$ is positive and is of class $C^2$.
\item $\phi h^n$ is absolutely integrable over $D$, $n=1,2,\cdots$.
\item $h$ has an absolutely maximum value at an interior point $\xi$ of $D$ and $\operatorname{det}\left(-\nabla^2 h(\xi)\right)>0$.
\item $\phi$ is continuous at $\xi$ and $\phi(\xi)\neq 0$.
\end{enumerate}
Then, we have
\[
\int_D \phi(x)[h(x)]^n dx\sim \left(\frac{2\pi}{n}\right)^\frac{d}{2}\frac{\phi(\xi)[h(\xi)]^n}{\sqrt{\operatorname{det}\left(-\nabla^2\log h(\xi)\right)}},\qquad n\to\infty.
\]
\end{lemma}
\begin{proof}
We refer to Lemma 1 in \cite{Hsu} for details.
\end{proof}
\begin{proposition}\label{P3.1}
Suppose that the law of $X_{in}$ is absolutely continuous with respect to the Lebesgue measure on $\bbr^d$ and let $f$ be a probability density function. Suppose that
\begin{enumerate}
\item $f$ is compactly supported.
\item $f$ is continuous at the global minimizer of $L$, namely $X_*$, and $f(X_*)>0$.
\item $L$ is $C^2$ and $\operatorname{det}\left(\nabla^2 L(X_*)\right)>0.$
\end{enumerate}
Then, we have
\[
-\frac{1}{\beta}\log\mathbb Ee^{-\beta L(X_{in})} = L(X_*)+ \frac{d}{2}\frac{\log\beta}{\beta}+ {\mathcal O}\left(\frac{1}{\beta} \right),\qquad \beta\to\infty.
\]
\end{proposition}
\begin{proof}
Let $D$ be any closed ball containing the support of $f$. We set
\[ \phi(x):=f(x) \quad \mbox{and} \quad h(x):=e^{-L(x)}. \]
For the above pair $(\phi, h)$, we will check that the assumptions in Lemma \ref{L3.3} hold. Note that
\begin{align*}
\begin{aligned}
\operatorname{det}\left(-\nabla^2 h(X_*)\right) &=\operatorname{det}\left(e^{-L(X_*)}\big(\nabla^2 L(X_*)-(\nabla L\otimes\nabla L)(X_*)\big)\right) \\
&=e^{-dL(X_*)}\operatorname{det}\left(\nabla^2 L(X_*) \right)>0.
\end{aligned}
\end{align*}
The rest of the assumptions can be checked straightforwardly. By Lemma \ref{L3.3}, we have
\begin{equation} \label{New-2}
\lim_{\beta\to\infty}\frac{\sqrt{\operatorname{det}\left(\nabla^2 L(X_*)\right)}\mathbb Ee^{-\beta L(X_{in})}}{\left(\frac{2\pi}{\beta}\right)^\frac{d}{2}f(X_*)e^{-\beta L(X_{*})}}=1.
\end{equation}
By taking logarithms on both sides of \eqref{New-2}, we get
\[
\log\mathbb Ee^{-\beta L(X_{in})}-\frac{d}{2}\log \left(\frac{2\pi}{\beta}\right)-\log f(X_*)+\beta L(X_*)+\frac{1}{2 }\log\operatorname{det}\left( \nabla^2 L(X_*)\right)={\mathcal O}(1).
\]
We multiply $-\frac{1}{\beta}$ on both sides to obtain
\[
-\frac{1}{\beta}\log\mathbb Ee^{-\beta L(X_{in})} = L(X_*)+\left(-\frac{d}{2}\log \left( 2\pi \right)-\log f(X_*)+\frac{1}{2 }\log\operatorname{det}\left( \nabla^2 L(X_*)\right)\right)\frac{1}{\beta}+\frac{d}{2}\frac{\log\beta}{\beta}+{\mathcal O}\left(\frac{1}{\beta}\right).
\]
Hence, one has the desired estimate:
\[
-\frac{1}{\beta}\log\mathbb Ee^{-\beta L(X_{in})} = L(X_*)+ \frac{d}{2}\frac{\log\beta}{\beta}+{\mathcal O}\left(\frac{1}{\beta} \right),
\]
for a sufficiently large $\beta$.
\end{proof}

\noindent $\bullet$~{\bf (Proof of Theorem \ref{T3.2}))}:~First, we claim:
\begin{equation} \label{New-3}
\essinf\limits_{\omega\in\Omega} L(X_\infty)\leq-\frac{1}{\beta}\log \mathbb E e^{-\beta L(X_{in})}-\frac{1}{\beta}\log\varepsilon.
\end{equation}
Note that
\begin{align}
\begin{aligned} \label{C-6}
& \frac{1}{N}\sum_{i=1}^N e^{-\beta L(X_{n+1}^i)}-\frac{1}{N}\sum_{i=1}^N e^{-\beta L(X_n^i)}  \\
& \hspace{0.5cm} = \frac{1}{N}\sum_{i=1}^N  e^{-\beta L(X_n^i)}\big(e^{-\beta (L(X_{n+1}^i)- L(X_n^i))}-1\big)\\
&  \hspace{0.5cm} \geq \frac{1}{N}\sum_{i=1}^N  e^{-\beta L(X_n^i)}(-\beta)\big(L(X_{n+1}^i)- L(X_n^i)\big) \\
& \hspace{0.5cm} =-\frac{\beta}{N}\sum_{i=1}^N  e^{-\beta L(X_n^i)}\nabla L(cX_{n+1}^i+(1-c)X_n^i)\cdot(X_{n+1}^i-X_n^i),
\end{aligned}
\end{align}
where we used a mean-value theorem with $c\in (0,1)$. We use definition of $\bar X_n^*$ to see 
\[ \bigg(\sum_{i=1}^N e^{-\beta L(X_n^i)}\bigg)\bar X_n^*=\sum_{i=1}^N e^{-\beta L(X_n^i)}X_n^i,  \quad 
 \bigg(\sum_{i=1}^N e^{-\beta L(X_n^i)}\bigg)\bar x_n^{*,l}=\sum_{i=1}^N e^{-\beta L(X_n^i)}x_n^{i,l},\qquad l=1,\cdots,d.
\]
This yields
\begin{align} \label{C-7a}
\begin{aligned}
  \sum_{i=1}^N& e^{-\beta L(X_n^i)}\nabla L(\bar X_n^*)\cdot (X_{n+1}^i-X_n^i)\\
&=\sum_{i=1}^N e^{-\beta L(X_n^i)}\nabla L(\bar X_n^*)\cdot \bigg( -\gamma(X^i_n - {\bar X}_n^*)  -\sum_{l=1}^{d} (x^{i,l}_n - {\bar x}_n^{*,l}) \eta_n^l e_l\bigg)\\
&=\sum_{i=1}^N e^{-\beta L(X_n^i)}\nabla L(\bar X_n^*)\cdot \bigg(  -\sum_{l=1}^{d} (x^{i,l}_n - {\bar x}_n^{*,l}) (\gamma+\eta_n^l )e_l\bigg) =0.
 \end{aligned}
\end{align}
Combining \eqref{C-6} and \eqref{C-7a} gives
\begin{align} 
\begin{aligned} \label{C-7}
& \frac{1}{N}\sum_{i=1}^N e^{-\beta L(X_{n+1}^i)}-\frac{1}{N}\sum_{i=1}^N e^{-\beta L(X_n^i)}\\
& \geq-\frac{\beta}{N}\sum_{i=1}^N  e^{-\beta L(X_n^i)}\nabla L(cX_{n+1}^i+(1-c)X_n^i)\cdot(X_{n+1}^i-X_n^i)\\
&  =-\frac{\beta}{N}\sum_{i=1}^N  e^{-\beta L(X_n^i)}\big(\nabla L(cX_{n+1}^i+(1-c)X_n^i)-\nabla L(\bar X_n^*)\big)\cdot(X_{n+1}^i-X_n^i) \\
& =-\frac{\beta}{N}\sum_{i=1}^N  e^{-\beta L(X_n^i)} \\
& \times \bigg[\int_0^1 \nabla^2 L\big(s(cX_{n+1}^i+(1-c)X_n^i)+(1-s)\bar X_n^*\big)\cdot\big(cX_{n+1}^i+(1-c)X_n^i-\bar X_n^*\big) ds \cdot(X_{n+1}^i-X_n^i)\bigg] \\
& \geq-\frac{\beta}{N}\sum_{i=1}^N  e^{-\beta L(X_n^i)}\Bigg[\int_0^1 \Big\|\nabla^2 L\big(s(cX_{n+1}^i+(1-c)X_n^i)+(1-s)\bar X_n^*\big)\Big\|_2ds \\
&\hspace{4cm} \times \big|cX_{n+1}^i+(1-c)X_n^i-\bar X_n^*\big||X_{n+1}^i-X_n^i|\Bigg] \\
&\geq-\frac{C_L\beta  e^{-\beta L_m}  }{N}\sum_{i=1}^N \big|cX_{n+1}^i+(1-c)X_n^i-\bar X_n^*\big||X_{n+1}^i-X_n^i|  \\
&=-\frac{C_L\beta  e^{-\beta L_m} }{N}\sum_{i=1}^N  \bigg| \sum_{l=1}^{d} (x^{i,l}_n - {\bar x}_n^{*,l}) (1-c\gamma-c\eta_n^l )e_l\bigg| \bigg| -\sum_{l=1}^{d} (x^{i,l}_n - {\bar x}_n^{*,l}) (\gamma +\eta_n^l )e_l\bigg| \\
&=-\frac{C_L\beta  e^{-\beta L_m} }{N}\sum_{i=1}^N  \sqrt{ \sum_{l=1}^{d} (x^{i,l}_n - {\bar x}_n^{*,l})^2 (1-c\gamma-c\eta_n^l )^2} \sqrt{ \sum_{l=1}^{d} (x^{i,l}_n - {\bar x}_n^{*,l})^2 (\gamma +\eta_n^l )^2}.
\end{aligned}
\end{align}
Set
\[ g(r):=(1-r\gamma-r\eta_n^l)^2, \quad 0\leq r\leq 1. \]
Then, we use the convexity of $g$ to find
\[
(1-c\gamma-c\eta_n^l)^2\leq \sup_{0\leq r\leq 1}g(r)=\max\{g(0),g(1)\}\leq g(0)+g(1)=1+(1-\gamma-\eta_n^l )^2.
\]
We substitute the above relation into \eqref{C-7} to obtain
\begin{align}
\begin{aligned} \label{C-8}
& \frac{1}{N}\sum_{i=1}^N e^{-\beta L(X_{n+1}^i)}-\frac{1}{N}\sum_{i=1}^N e^{-\beta L(X_n^i)}\\
& \hspace{0.2cm} \geq -\frac{C_L\beta  e^{-\beta L_m} }{N}\sum_{i=1}^N  \sqrt{ \sum_{l=1}^{d} (x^{i,l}_n - {\bar x}_n^{*,l})^2 \big(1+(1-\gamma-\eta_n^l )^2\big)} \sqrt{ \sum_{l=1}^{d} (x^{i,l}_n - {\bar x}_n^{*,l})^2 (\gamma +\eta_n^l )^2}.
\end{aligned}
\end{align}
Taking expectations on the both sides of \eqref{C-8} using the Cauchy-Schwarz inequality leads to
\begin{align}
\begin{aligned} \label{C-9}
\begin{split}
& \frac{1}{N}\sum_{i=1}^N\mathbb E e^{-\beta L(X_{n+1}^i)}-\frac{1}{N}\sum_{i=1}^N \mathbb Ee^{-\beta L(X_n^i)} \\
&\hspace{0.2cm} \geq -\frac{C_L\beta  e^{-\beta L_m} }{N}\sum_{i=1}^N  \mathbb E\left[\sqrt{ \sum_{l=1}^{d} (x^{i,l}_n - {\bar x}_n^{*,l})^2 \big(1+(1-\gamma-\eta_n^l )^2\big)} \sqrt{ \sum_{l=1}^{d} (x^{i,l}_n - {\bar x}_n^{*,l})^2 (\gamma+\eta_n^l )^2} \right]\\
&\hspace{0.2cm}\geq -\frac{C_L\beta  e^{-\beta L_m} }{N}\sum_{i=1}^N \sqrt{ \mathbb E\left[ \sum_{l=1}^{d} (x^{i,l}_n - {\bar x}_n^{*,l})^2 \big(1+(1-\gamma-\eta_n^l )^2\big)\right]\mathbb E \left[ \sum_{l=1}^{d} (x^{i,l}_n - {\bar x}_n^{*,l})^2 (\gamma+\eta_n^l )^2\right] }\\
&\hspace{0.2cm}=-\frac{C_L\beta  e^{-\beta L_m} }{N}\sum_{i=1}^N \sqrt{ \bigg( \sum_{l=1}^{d} \mathbb E|x^{i,l}_n - {\bar x}_n^{*,l}|^2\mathbb E \big[1+(1-\gamma-\eta_n^l )^2\big]\bigg)\bigg(\sum_{l=1}^{d} \mathbb E|x^{i,l}_n - {\bar x}_n^{*,l}|^2\mathbb E |\gamma +\eta_n^l |^2\bigg) }\\
&\hspace{0.2cm}=-\frac{C_L\beta  e^{-\beta L_m} }{N}\sum_{i=1}^N \sqrt{ \Big(\big(1+ (1-\gamma)^2+\zeta^2\big)\mathbb E|X^i_n - {\bar X}_n^*|^2\Big)\Big(\big( \gamma^2+\zeta^2\big)\mathbb E|X^i_n - {\bar X}_n^*|^2\Big) }\\
&\hspace{0.2cm}=-\frac{C_L \sqrt{ \big(1+ (1-\gamma)^2+\zeta^2\big)\big( \gamma^2+\zeta^2\big)}\beta  e^{-\beta L_m} }{N}\sum_{i=1}^N \mathbb E|X^i_n - {\bar X}_n^*|^2\\
&\hspace{0.2cm}\geq  -2C_L \sqrt{ \big(1+ (1-\gamma)^2+\zeta^2\big) \big( \gamma^2+\zeta^2\big)}e^{-(2\delta  - \gamma^2 -\zeta^2)n} \beta  e^{-\beta L_m}\sum_{l=1}^d \left(\mathbb E\max_{1\leq i\leq N} (x_0^{i,l} -{\bar x}^l_0)^2\right).
\end{split}
\end{aligned}
\end{align}
We sum up \eqref{C-9} over $n$ to get 
\begin{align*}
\begin{aligned} \label{C-10}
&\frac{1}{N}\sum_{i=1}^N \mathbb E e^{-\beta L(X_n^i)} \geq \frac{1}{N}\sum_{i=1}^N \mathbb E e^{-\beta L(X_0^i)} \\
& \hspace{1cm} -\frac{2C_L \sqrt{ \big(1+ (1-\gamma)^2+\zeta^2\big) \big( \gamma^2+\zeta^2\big)} \beta  e^{-\beta L_m}}{1-e^{-(2\gamma  - \gamma^2 -\zeta^2)}} \sum_{l=1}^d \left(\mathbb E\max_{1\leq i\leq N} (x_0^{i,l} -{\bar x}^l_0)^2\right).
\end{aligned}
\end{align*}
Letting $n\to\infty$, we use Lemma \ref{L3.2}  to find
\begin{align*}
\mathbb Ee^{-\beta L(X_\infty)} &\geq \frac{1}{N}\sum_{i=1}^N \mathbb E e^{-\beta L(X_0^i)} \\
&-\frac{2C_L \sqrt{ \big(1+ (1-\gamma)^2+\zeta^2\big) \big( \gamma^2+\zeta^2\big)} \beta  e^{-\beta L_m}}{1-e^{-(2\gamma  - \gamma^2 -\zeta^2)}} \sum_{l=1}^d \left(\mathbb E\max_{1\leq i\leq N} (x_0^{i,l} -{\bar x}^l_0)^2\right)\\
& \geq \varepsilon\mathbb E e^{-\beta L(X_{in})}.
\end{align*}
Hence
\[e^{-\beta \essinf\limits_{\omega\in\Omega} L(X_\infty)} =\mathbb Ee^{-\beta \essinf\limits_{\omega\in\Omega}L(X_\infty)}\geq\mathbb Ee^{-\beta L(X_\infty)}\geq  \varepsilon\mathbb E e^{-\beta L(X_{in})}. \]
By taking logarithm to the both sides of the above relation yields  the desired estimate \eqref{New-3}. \newline

\noindent Next, we apply Proposition \ref{P3.1} to obtain
\[
\essinf\limits_{\omega\in\Omega} L(X_\infty)\leq L(X_*)+ \frac{d}{2}\frac{\log\beta}{\beta}+E(\beta),
\]
for some function $E(\beta)= O\left(\frac{1}{\beta} \right)$. This completes the proof of Theorem \ref{T3.2}.
\section{Conclusion} \label{sec:4}
\setcounter{equation}{0}
In this paper, we presented a convergence study with an error estimate for a time-discrete consensus-based non-convex optimization algorithm, which includes three discrete-in-time algorithms used in \cite{C-J-L-Z}. For the continuous algorithm introduced in \cite{C-J-L-Z}, convergence and error analysis was studied  by the authors recently using the detailed structure of the algorithm and It\^o's stochastic analysis. However,  for the discrete algorithm such analysis was not done in the aforementioned work \cite{H-J-K}. The main reason for that  was the lack of discrete analog of the It\^o's stochastic analysis which has been crucially used in the study  of the continuous algorithm. Our error analysis employed in this work is based on the elementary and simple probability arguments together with the detailed structure of nonlinear interaction terms, and consequently  provides 
the stability condition, convergennce and an error analysis toward the global minumum for all three time-discrete algorithms implemented in \cite{C-J-L-Z}. Our analysis shows that these algorithms, under certain conditions on the parameters and with sufficiently well-chosen initial data,  do converge to
a point close to the global minimum for moderately high dimensional problems. 

\newpage

\appendix

\section{Error estimate under an alternative framework} \label{App-A}
\setcounter{equation}{0}
In this subsection, we provide an alternative result for an error estimate for time-discrete CBO scheme \eqref{A-1} without using Laplace's principle under a slightly different framework $({\mathcal B}1) - ({\mathcal B}2)$. Below, we present a framework for the objective function $L$, global minimum point $X_*$ and reference 
random variable $X_{in}$ as follows: \newline
\begin{itemize}
\item
$({\mathcal B}1)$:~Let $L = L(x)$ be a $C^2$-objective function  satisfying the following relations:
\begin{equation*} \label{C-5}
 L_m:=\min_{x \in \bbr^d} L(x) >0 \quad \mbox{and} \quad C_L:=\sup_{x\in\bbr^d}\|\nabla^2L(x)\|_2 <\infty,
\end{equation*} 
where $\|\cdot\|_2$ denotes the spectral norm. 

\vspace{0.2cm}
%
%\item
%Let $X_*$ be the unique global minimum point of $L$ in $\bbr^d$ satisfying the local convexity relation:
%\begin{equation*} \label{C-5-0}
%\operatorname{det}\left(\nabla^2 L(X_*)\right)>0.
%\end{equation*}
%
%\vspace{0.2cm}

\item 
$({\mathcal B}2)$:~Let $X_{in}$ be a reference random variable associated with a law whose support ${\tilde D}$ is compact and contains $X_*$.\end{itemize}
Note that the condition $({\mathcal B}1)$ is exactly the same as $({\mathcal A}1)$, whereas the condition $({\mathcal B}2)$ is different from $({\mathcal A}3)$ in which the probability measures associated with $X_{in}$ is absolutely continuous with respect to Lebesgue measure. Moreover, notice that this new framework does not need any condition on $\operatorname{det}\left(\nabla^2 L(X_*)\right)$ as in $({\mathcal A}2)$. 

\vspace{0.2cm}

Next, we study how the common consensus state $X_\infty$ is close to the global minimum $X_*$ in suitable sense.
Now, we are ready to provide an error estimate for  the discrete CBO algorithm, which is analogous to Theorem 4.1 in \cite{H-J-K} for the continuous case. 
\begin{theorem} \label{TA.1}
Suppose that the framework $({\mathcal  B}1) - ({\mathcal  B}2)$ holds, and parameters $\beta, \gamma, \zeta, \delta$ and the initial data $\{X_0^i \}$ satisfy 
\begin{align}
\begin{aligned} \label{AP-1}
& \beta > 0, \quad \delta > 0, \quad  (\gamma-1)^2 + \zeta^2 < 1, \quad X_0^i: i,i.d, \quad X_0^i \sim X_{in},  \quad \sup_{x \in {\tilde {\mathcal D}}} L(x) - L_m  < \delta, \\
& (1-\varepsilon)\mathbb E \Big[ e^{-\beta L(X_{in})} \Big ] \\
& \hspace{0.5cm} \geq \frac{2C_L \sqrt{ \big(1+ (1-\gamma)^2+\zeta^2\big) \big( \gamma^2+\zeta^2\big)} \beta  e^{-\beta L_m}}{1-e^{-[ 1-(\gamma - 1)^2 - \zeta^2]}} \sum_{l=1}^d \left(\mathbb E\max_{1\leq i\leq N} (x_0^{i,l} -{\bar x}^l_0)^2\right),
\end{aligned}
\end{align}
for some $0<\varepsilon<1$. Then, for a solution $\{X_n^i\}_{1\leq i\leq N}$ to \eqref{A-1}, 
\begin{equation} \label{AP-2}
\Big| \essinf\limits_{\omega\in\Omega} L(X_\infty) - L_m \Big | \leq \delta + \Big| \frac{\log\varepsilon}{\beta} \Big|.
\end{equation}
\end{theorem}
\begin{remark}\label{R3.1}
Before we provide a proof, we give several comments on the result of this theorem.
%1. Note that the third condition $\eqref{C-5-2}_2$:
%\[
%(1-\varepsilon)\mathbb E \Big[ e^{-\beta L(X_{in})} \Big ] \geq \frac{2C_L \sqrt{ \big(1+ (1-\gamma)^2+\zeta^2\big) \big( \gamma^2+\zeta^2\big)} \beta  e^{-\beta L_m}}{1-e^{-[ 1-(\gamma - 1)^2 - \zeta^2]}} \sum_{l=1}^d \left(\mathbb E\max_{1\leq i\leq N} (x_0^{i,l} -{\bar x}^l_0)^2\right),
%\]
%does hold only for some intermediate $\beta$ for the generic initial data satisfying
%\[ \sum_{l=1}^d \left(\mathbb E\max_{1\leq i\leq N} (x_0^{i,l} -{\bar x}^l_0)^2\right) > 0. \]
%This can be checked as follows. We multiply $e^{\beta L_m}$ to get
%\[ (1-\varepsilon)\mathbb E \Big[ e^{\beta(L_m - L(X_{in}))} \Big ] \geq \frac{2C_L \sqrt{ \big(1+ (1-\gamma)^2+\zeta^2\big) \big( \gamma^2+\zeta^2\big)} \beta }{1-e^{-(2\gamma  - \gamma^2 -\zeta^2)}} \sum_{l=1}^d \left(\mathbb E\max_{1\leq i\leq N} (x_0^{i,l} -{\bar x}^l_0)^2\right). \]
%By letting $\beta \to \infty$ in the above relation, we can see that LH.S. is less than or equal to $1-\varepsilon$, but R.H.S. goes to $+\infty$. Hence, 
1. In the proof of Theorem \ref{TA.1}, we will first derive the estimate:
\begin{equation} \label{AP-2-1}
\essinf\limits_{\omega\in\Omega} L(X_\infty)\leq \sup\limits_{x\in {\tilde D}} L(x)-\frac{1}{\beta}\log\varepsilon.
\end{equation}
2. For any given $\delta > 0,~\varepsilon > 0$ and $\beta > 0$, the conditions $\eqref{AP-1}$ can be attained with suitable $X_{in}$. To see this, we choose the law of $X_{in}$ such that ${\tilde D}$ is a small neighborhood of a global minimum $X_*$  satisfying the following two relations:
\begin{equation} \label{AP-3-1}
\sup\limits_{x\in {\tilde D}} L(x)-L_m<\delta,
\end{equation}
and
\begin{equation} \label{AP-4}
1-\varepsilon\geq \frac{2C_L \sqrt{ \big(1+ (1-\gamma)^2+\zeta^2\big) \big( \gamma^2+\zeta^2\big)} \beta  e^{\beta (\sup\limits_{x\in D}L(x)- L_m)}}{1-e^{-[ 1-(\gamma - 1)^2 - \zeta^2]}} \big( \operatorname{diam}({\mathcal R}_{in})\big)^2,
\end{equation}
 where ${\mathcal R}_{in} = [a_1, b_1] \times \cdots \times [a_d, b_d]$ is the  smallest closed $d$-dimensional rectangle containing ${\tilde D}$ so that 
\begin{equation} \label{AP-3}
\mathbb E \Big[ \max_{1\leq i\leq N} (x_0^{i,l} -{\bar x}^l_0)^2 \Big] \leq |b_\ell - a_\ell|^2,  
\end{equation}
for each $\ell = 1, \cdots, d$.  Then, due to \eqref{AP-3}, the relation \eqref{AP-4} implies the condition $\eqref{AP-1}_2$. Hence, we can apply the estimate \eqref{AP-2-1} and \eqref{AP-3-1} to get the desired error estimate:
\[
\essinf\limits_{\omega\in\Omega} L(X_\infty)\leq L_m+\left( \sup\limits_{x\in D} L(x)-L_m\right)-\frac{1}{\beta}\log\varepsilon<L_m+\delta-\frac{\log\varepsilon}{\beta}.
\]

\end{remark}

\vspace{0.5cm}
Now we are ready to provide a proof of Theorem \ref{TA.1}. By the same way as in the proof of Theorem \ref{T3.2}, we obtain%Note that
\begin{align*}
\mathbb Ee^{-\beta L(X_\infty)} &\geq \frac{1}{N}\sum_{i=1}^N \mathbb E e^{-\beta L(X_0^i)} \\
&-\frac{2C_L \sqrt{ \big(1+ (1-\gamma)^2+\zeta^2\big) \big( \gamma^2+\zeta^2\big)} \beta  e^{-\beta L_m}}{1-e^{-(2\gamma  - \gamma^2 -\zeta^2)}} \sum_{l=1}^d \left(\mathbb E\max_{1\leq i\leq N} (x_0^{i,l} -{\bar x}^l_0)^2\right)\\
& \geq \varepsilon\mathbb E e^{-\beta L(X_{in})}.
\end{align*}
Hence one has
\[e^{-\beta \essinf\limits_{\omega\in\Omega} L(X_\infty)} =\mathbb Ee^{-\beta \essinf\limits_{\omega\in\Omega}L(X_\infty)}\geq\mathbb Ee^{-\beta L(X_\infty)}\geq  \varepsilon\mathbb E e^{-\beta L(X_{in})}\geq\varepsilon e^{-\beta \sup\limits_{x\in D} L(x)}. \]
Finally, we take logarithm to the both sides of the above relation to get the desired estimate \eqref{AP-2-1}.


\begin{thebibliography}{00}

\bibitem{A-B} Acebron, J. A., Bonilla, L. L., P\'{e}rez Vicente, C. J. P., Ritort, F. and Spigler, R.: \textit{The Kuramoto model: A simple paradigm for synchronization phenomena.} Rev. Mod. Phys. \textbf{77} (2005), 137-185.

%\bibitem{A-H} Ahn, S. M. and Ha, S.-Y.: \textit{Stochastic flocking dynamics of the Cucker-Smale model with multiplicative white noises.} J. Math. Physics. {\bf 51} (2011), 103301.

\bibitem{A-B-F} Albi, G., Bellomo, N., Fermo, L., Ha, S.-Y., Pareschi, L., Poyato, D. and Soler, J.: \textit{Vehicular traffic, crowds, and swarms. On the kinetic theory approach towards research perspectives.} Math. Models Methods Appl. Sci. {\bf 29} (2019), 1901-2005.


%\bibitem{A-P} Albi, G. and Pareschi, L.: \textit{Binary interaction algorithms for the simulation of flocking and swarming dynamics.} Multiscale Modeling and Simulation {\bf 11}, 129 (2013).

%\bibitem{B-H} Bellomo, N. and Ha, S.-Y.: \textit{A quest toward a mathematical theory of the dynamics of swarms}, Math. Models Methods Appl. Sci. {\bf 27} (2017), 745-770.

%\bibitem{Bo} Bottou, L.: \textit{Online learning and stochastic approximations.} On-line learning in neural networks {\bf 17}, 142 (1998).


\bibitem{Be} Bertsekas, D.: \textit{Convex analysis and optimization.} Athena Scientific. 2003.

%\bibitem{B-B} Buck, J. and Buck, E.: \textit{Biology of synchronous flashing of fireflies.} Nature {\bf211} (1966), 562-564.

\bibitem{C-J-L-Z} Carrillo, J. A., Jin, S., Li, L. and Zhu, Y.: \textit{A consensus-based global optimization method for high dimensional machine learning problems.} Submitted.

\bibitem{C-C-T-T} Carrillo, J., Choi, Y.-P., Totzeck, C. and Tse, O.: \textit{An analytical framework for consensus-based global optimization method.} Mathematical Models and Methods in Applied Sciences {\bf 28} (2018), 1037-1066.

%\bibitem{C-F-R-T}  Carrillo, J. A., Fornasier, M., Rosado, J. and Toscani, G.: \textit{Asymptotic flocking dynamics for the kinetic Cucker-Smale model.} SIAM. J. Math. Anal. {\bf 42} (2010), 218-236. 

%\bibitem{C-D-P} Cattiaux, P., Delebecque, F. and Pedeches, L.: \textit{Stochastic Cucker-Smale models: old and new.} Ann. Appl. Probab. {\bf28} (2018), 3239-3286. 

%\bibitem{C-H-L} Choi, Y.-P. and Ha, S.-Y.: \textit{A simple proof of the complete consensus of discrete-time dynamical networks with time-varying couplings}. International J. of Numerical Analysis and Modeling, series B. {\bf 1},  58-69, 2010.

\bibitem{C-H-L} Choi, Y.-P., Ha, S.-Y. and Li, Z.: \textit{Emergent dynamics of the Cucker-Smale flocking model and its variants.} In N. Bellomo, P. Degond, and E. Tadmor (Eds.), Active Particles Vol.I Theory, Models, Applications (tentative title), Series: Modeling and Simulation in Science and Technology, Birkhauser, Springer. 2017.

%\bibitem{D-Z} Dai, X. and Zhu, Y.: \textit{Towards theoretical understanding of large batch training in stochastic gradient descent.} arXiv:1812.00542 (2018).

\bibitem{C-S2} Crow, E. L. and Shimizu, K., eds.: \textit{Lognormal distributions: theory and applications.} Statistics: Textbooks and Monographs, {\bf 88}, New York: Marcel-Dekker, Inc., 1988.

\bibitem{C-S} Cucker, F. and Smale, S.: \textit{On the mathematics of emergence.} Japanese Journal of Mathematics {\bf 2} (2007), 197-227.

\bibitem{D-Z} Dembo, A. and Zeitouni, O.: \textit{Large deviations techniques and applications.} Springer-Verlag, Berlin, Heidelberg, second edition, 1998.

%\bibitem{D-Z2} Ding, J. and Zhou, A.: \textit{Eigenvalues of rank-one updated matrices with some applications.} Appl. Math. Lett. {\bf 20} (2007), 1223-1226.

\bibitem{E-K} Eberhart, R. and Kennedy, J.: \textit{Particle swarm optimization.} Proceedings of the IEEE International Conference on Neural Networks {\bf 4} (1995), 1942-1948.

\bibitem{F-H-J} Fang, D., Ha, S.-Y. and Jin, S.: \textit{Emergent behaviors of the Cucker-Smale ensemble under attractive-repulsive couplings and Rayleigh frictions.} Math. Models Methods Appl. Sci. {\bf 29} (2019), 1349-1385.

\bibitem{F-H-P-S1} Fornasier, M., Huang, H., Pareschi, L. and S$\ddot{u}$nnen, P.: \textit{Consensus-based optimization on the sphere I: Well-posedness and mean-field limit.} Preprint. Available at arXiv:2001.110994v2.

\bibitem{F-H-P-S2} Fornasier, M., Huang, H., Pareschi, L. and S$\ddot{u}$nnen, P.: \textit{Consensus-based optimization on the sphere II: Convergence to global minimizer and machine learning.} Preprint. Available at arXiv:2001.11988v2.

\bibitem{H-J-K} Ha, S.-Y., Jin, S. and Kim, D.: \textit{Convergence of a first-order consensus-based global optimization algorithm}. Submitted. arxiv: 1910.08239.

\bibitem{H-L} Ha, S.-Y. and Liu, J.-G.: \textit{A simple proof of Cucker-Smale flocking dynamics and mean-field limit.} Commun. Math. Sci. {\bf 7} (2009), 297-325. 

\bibitem{H-L-L}  Ha, S.-Y., Lee, K. and Levy, D.: \textit{Emergence of time-asymptotic flocking in a stochastic Cucker-Smale system.} Commun. Math. Sci. {\bf7} (2009), 453-469. 

%\bibitem{Ha} Hanin, B.: \textit{Which neural net structures give rise to exploding and vanishing gradients?} Advances in Neural Information Processing Systems. 582-591 (2018).

\bibitem{Ho} Holland, J. H.: \textit{Genetic algorithms.} Scientific American {\bf 267} (1992), 66-73.

\bibitem{Hsu} Hsu, L. C.: \textit{A theorem on the asymptotic behavior of a multiple integral.} Duke Math. J. {\bf 15} (1948), 623-632.

%\bibitem{J-L-L} Jin, S. Li, L. and Liu, J.-G.: \textit{Random batch methods(RBM) for interacting particle systems.} J. Comp. Phys., to appear. arXiv:1812.10575, 2018.

\bibitem{Ke} Kennedy, J.: \textit{Swarm intelligence}. Handbook of nature-inspired and innovative computing. Springer 2006, 187-219.

\bibitem{K-G-V} Kirkpatrick, S., Gelatt, C. D. and Vecchi, M. P.: \textit{Optimization by simulated annealing.} Science {\bf 220} (1983), 671-680.

\bibitem{K-C-B-F-L} Kolokolnikov, T., Carrillo, J. A., Bertozzi, A., Fetecau, R. and Lewis, M.: \textit{Emergent behavior in a multi-particle systems with non-local interactions.} Physica D {\bf 260} (2013), 1-4.

\bibitem{Ku1} Kuramoto, Y.: \textit{Chemical oscillations, waves and turbulence.} Springer-Verlag, Berlin, 1984.

\bibitem{Ku2} Kuramoto, Y.: \textit{International symposium on mathematical problems in mathematical physics.} Lecture Notes Theor. Phys.  \textbf{30} (1975), 420.

\bibitem{L-A} Laarhoven, P. J. M. van and Aarts, E. H. L.: \textit{Simulated annealing: theory and applications.} D. Reidel Publishing Co., Dordrecht, 1987.

\bibitem{M-T} Motsch, S. and Tadmor, E.: \textit{Heterophilious dynamics enhances consensus.} SIAM. Rev. {\bf 56} (2014), 577-621.

\bibitem{Pe} Peskin, C. S.: \textit{Mathematical aspects of heart physiology.} Courant Institute of Mathematical Sciences, New York, 1975.

\bibitem{P-T-T-M} Pinnau, R., Totzeck, C., Tse, O. and Martin, S.: \textit{A consensus-based model for global optimization and its mean-field limit}. Math. Models Methods Appl. Sci. {\bf 27} (2017), 183-204.

\bibitem{P-R-K}  Pikovsky, A., Rosenblum, M. and Kurths, J.:  \textit{Synchronization: A universal concept in nonlinear sciences.} Cambridge University Press, Cambridge, 2001.

%\bibitem{To} Toscani, G.: \textit{Kinetic models of opinion formation.} Commun. Math. Sci. {\bf 4}, 481-496 (2006).

\bibitem{T-P-B-S} Totzeck, C. Pinnau, R., Blauth, S. and Schotth\'{o}fer, S.: \textit{A numerical comparison of consensus-based global optimization to other particle-based global optimization scheme.} Proceedings in Applied Mathematics and Mechanics, {\bf 18}, 2018.

\bibitem{V-Z} Vicsek, T. and Zefeiris, A.: \textit{Collective motion.} Phys. Rep. {\bf 517} (2012), 71-140.

\bibitem{Ya} Yang, X.-S.: \textit{Nature-inspired metaheuristic algorithms.} Luniver Press, 2010.
\bibitem{Y-D} Yang, X.-S., Deb, S., Zhao, Y.-X., Fong, S. and He, X.: \textit{Swarm intelligence: past, present and future}. Soft Comput {\bf 22} (2018), 5923-5933. 

\end{thebibliography}
\end{document}